\documentclass[12pt]{amsart}

\usepackage[a4paper,margin=1in]{geometry}
\usepackage[T1]{fontenc}
\usepackage{amsmath,amssymb,amsthm,mathtools}
\usepackage{microtype}
\usepackage{xcolor}
\definecolor{linkblue}{RGB}{0,82,155}
\usepackage[
  colorlinks=true,
  linkcolor=linkblue,
  citecolor=linkblue,
  urlcolor=linkblue
]{hyperref}
\usepackage[nameinlink,capitalise,noabbrev]{cleveref}
\usepackage{enumitem}
\usepackage{placeins}

\crefname{theorem}{Theorem}{theorems}
\Crefname{theorem}{Theorem}{Theorems}

\crefname{lemma}{lemma}{lemmas}
\Crefname{lemma}{Lemma}{Lemmas}

\crefname{corollary}{corollary}{corollaries}
\Crefname{corollary}{Corollary}{Corollaries}

\crefname{proposition}{proposition}{propositions}
\Crefname{proposition}{Proposition}{Propositions}

\crefname{remark}{remark}{remarks}
\Crefname{remark}{Remark}{Remarks}

\crefname{definition}{definition}{definitions}
\Crefname{definition}{Definition}{Definitions}

\selectfont
\setlist{itemsep=0.25em,topsep=0.4em}

\usepackage{aliascnt}

\newtheorem{theorem}{Theorem}[section]

\newaliascnt{lemma}{theorem}
\newtheorem{lemma}[lemma]{Lemma}
\aliascntresetthe{lemma}

\newaliascnt{corollary}{theorem}
\newtheorem{corollary}[corollary]{Corollary}
\aliascntresetthe{corollary}

\newaliascnt{proposition}{theorem}
\newtheorem{proposition}[proposition]{Proposition}
\aliascntresetthe{proposition}

\newtheoremstyle{myremark}%
  {6pt}% Space above
  {6pt}% Space below
  {\normalfont}% Body font
  {}% Indent amount
  {\bfseries}% Head font
  {.}% Punctuation after head
  {.5em}% Space after head
  {}% Head spec

\theoremstyle{myremark}
\newaliascnt{remark}{theorem}
\newtheorem{remark}[remark]{Remark}
\aliascntresetthe{remark}

\newaliascnt{definition}{theorem}
\newtheorem{definition}[definition]{Definition}
\aliascntresetthe{definition}

\DeclareMathOperator{\tr}{tr}
\DeclareMathOperator{\rank}{rank}

\newcommand{\R}{\mathbb{R}}
\newcommand{\1}{\mathbf{1}}
\newcommand{\inner}[2]{\langle #1,#2\rangle}
\newcommand{\froinner}[2]{\langle #1,#2\rangle_{\!F}}
\newcommand{\norm}[1]{\lVert #1\rVert}
\newcommand{\abs}[1]{\lvert #1\rvert}
\newcommand{\cP}{\mathcal{P}}

\title{Graph Eigenvalues and Projection Constants}

\author{Varun Sivashankar}
\address{Department of Mathematics, Princeton University, Princeton, New Jersey, USA}
\email{varunsiva@princeton.edu}

\author{Quanyu Tang}
\address{School of Mathematical Sciences, University of Science and Technology of China, Hefei 230026, P.\ R.\ China}
\email{tangquanyu827@gmail.com}

\author{Tanay Wakhare}
\address{Department of Electrical Engineering and Computer Science, Massachusetts Institute of Technology, Cambridge, Massachusetts, USA \newline \hspace*{1em} MathVision AI}
\email{tanay@mathvision.ai}
\date{}

\begin{document}
\begin{abstract}
For an integer $k\ge2$, let $\lambda_k(G)$ denote the $k$th largest adjacency eigenvalue of a graph $G$. For every graph $G$ on $n$ vertices and every $2 \leq k \leq n$, we prove
\[
\lambda_k(G)
\le \frac{(k-2)\sqrt{k+1}+2}{2k(k-1)}\,n-1.
\]
Our bound is tight for $k\in\{2,3,4,8,24\}$. We obtain it by reducing the graph-eigenvalue problem to an extremal problem for orthogonal projections and then applying the general upper bound on the absolute projection constant $\gamma(r)$ due to Der\k{e}gowska and Lewandowska. We also give an alternative proof of their bound by repairing the Gegenbauer-polynomial argument of K\"onig and Tomczak-Jaegermann. The resulting slack identity yields a strict improvement in every even dimension $r\ge4$ for which $r+2$ is not a perfect square.
\end{abstract}

\maketitle

\section{Introduction}

For a graph $G$ on $n$ vertices, let
\[
\lambda_1(G)\ge \lambda_2(G)\ge \cdots\ge \lambda_n(G)
\]
denote the eigenvalues of its adjacency matrix. More generally, if \(M\) is a real symmetric matrix of order \(m\), we write
\[
\lambda_1(M)\ge\lambda_2(M)\ge\cdots\ge\lambda_m(M)
\]
for its eigenvalues, counted with multiplicity. Following Nikiforov~\cite{Nikiforov15}, define
\[
c_k:=\sup\left\{\frac{\lambda_k(G)}{|V(G)|}: |V(G)|\ge k\right\}.
\]

The maximum-degree bound gives $\lambda_1(G)\leq n-1$, with equality for the
clique $K_n$. Hong~\cite{Hong88} proved that
$\lambda_k(G) \leq \frac{n-k}{2}$ for \(2\le k\le n\). In particular,
\(\lambda_2(G) \leq \frac{n}{2}-1\), and this is sharp when \(n\) is even,
as shown by \(K_{n/2}\sqcup K_{n/2}\). Hong then posed the problem of bounding
$\lambda_k(G)$ in terms of the order of the graph~\cite{Hong93}.

If $n$ is divisible by $k$ and $G = K_{n/k} \sqcup \cdots \sqcup K_{n/k}$ is the disjoint union of $k$ cliques each of size $\frac{n}{k}$, then $\lambda_k(G) = \frac{n}{k}-1$. This motivated a conjecture that $\lambda_k(G) \leq \lfloor\frac{n}{k}\rfloor$. Powers~\cite{Powers89} claimed that this upper bound holds, which would have essentially resolved the problem. Nikiforov~\cite{Nikiforov15} later found a gap in the proof and gave counterexamples for \(k\ge5\).

Nikiforov used the fact that for any odd prime power $q$, there exists a strongly regular graph on $q^3$ vertices~\cite{taylor1977regular}. The complements of these graphs are also strongly regular, and Nikiforov realized that whenever $k = q(q-1)+1$, the existence of this graph implies that $c_k \geq \frac{1}{2\sqrt{k-1}+1}$. Then using a result on the distribution of primes by Baker, Harman, and Pintz~\cite{baker2001difference}, Nikiforov proved that $c_k \geq \frac{1}{2\sqrt{k-1}+k^{1/3}}$ for all sufficiently large $k$. He also proved that $c_k \geq \frac{1}{4\sqrt{k-1}}$ for $k \geq 16$ and $c_k \geq \frac{1}{k-1/2}$ for $5 \leq k \leq 15$ using certain regular graphs. However, he asked whether it is still true that $c_3=1/3$ and $c_4=1/4$~\cite[Question~2.11]{Nikiforov15}. In 2023, Linz~\cite{Linz23} constructed examples showing that $c_4 \geq \frac{1+\sqrt{5}}{12} > \frac{1}{4}$ along with some improved constructions for $k \leq 24$. This left only the case $k=3$, and recent work by Leonida and Li~\cite{LeonidaLi25} and Li~\cite{Li25} provided strong evidence that the order bound $\frac{n}{3}$ might actually hold for the third eigenvalue problem.

Our main contribution is proving that $c_k$ can be upper bounded (up to scaling factors) by the absolute projection constant studied in Banach space theory. For $r\ge 1$, let
\[
\cP_r(N):=\{Q\in \R^{N\times N}: Q=Q^\top=Q^2,\ \rank(Q)=r\}.
\]
Let $q_{ij}$ denote the entries of $Q$. Define the entrywise \(\ell_1\)-norm of $Q$ by
\[\|Q\|_1 = \sum_{i=1}^N \sum_{j=1}^N |q_{ij}|.\]

For \(N\ge r\), define the finite-dimensional quantity
\[
\gamma(N,r):=\frac1N\max_{Q\in \cP_r(N)} \|Q\|_1, 
\]
and set
\[
\gamma(r):=\sup_{N\ge r}\gamma(N,r).
\]
We call \(\gamma(r)\) the absolute projection constant in dimension \(r\).
This orthogonal-projection formula agrees with the maximal absolute projection
constant of Banach space theory. See \cite{Basso19,DeregowskaLewandowska23}.

Our main result is the following.

\subsection*{\texorpdfstring{\Cref{thm:main-explicit}}{The main explicit bound}}
Let $2 \leq k \leq n$.
For every graph $G$ on $n$ vertices,
\[
\lambda_k(G)\le \frac{\gamma(k-1)}{2(k-1)}\,n-1
\le \frac{(k-2)\sqrt{k+1}+2}{2k(k-1)}\, n-1.
\]
In particular, $c_k\le \frac{(k-2)\sqrt{k+1}+2}{2k(k-1)}$.
We write
\[
\alpha_k:=\frac{(k-2)\sqrt{k+1}+2}{2k(k-1)}
\]
for this coefficient.

For $k=2$, this is the classical sharp bound $\lambda_2(G)\le n/2-1$~\cite{Hong88,Nikiforov15}. For $k=3$, it yields the sharp bound $\lambda_3(G)\le n/3-1$. For $k\in\{4,8,24\}$, the coefficient $\alpha_k$ matches Linz's lower bounds~\cite{Linz23}, so the theorem is tight for $k\in\{2,3,4,8,24\}$. This tightness comes from equiangular lines. See \Cref{sec:weighted-equiangular}.

The first inequality is due to our reduction from the graph eigenvalue problem to the absolute projection constant $\gamma(k-1)$ as described in \Cref{sec:reduction}. The second inequality follows from the absolute projection constant bound of Der\k{e}gowska and Lewandowska~\cite{DeregowskaLewandowska23}.

\subsection*{\texorpdfstring{\Cref{thm:projection-constant}~\cite{DeregowskaLewandowska23}}{The absolute projection constant bound}}
Let \(Q\in \R^{n\times n}\) be a rank-\(r\) orthogonal projection, where \(r\ge 1\). 

\[
\|Q\|_1=\sum_{i,j=1}^n \abs{q_{ij}}
\le \frac{r+\sqrt{r+2}}{1+\sqrt{r+2}}\, n.
\]
So in particular, 
\[\gamma(r) \leq \frac{r+\sqrt{r+2}}{1+\sqrt{r+2}}.\]

K\"onig and Tomczak-Jaegermann~\cite{KonigTomczak94} claimed a proof of this estimate, together with sharpness and uniqueness statements in the equiangular cases.  Their later paper~\cite{KonigTomczak03} revised part of the equality-case picture, but the argument for the general estimate is incorrect. Kobos~\cite{Kobos25} gives a careful account of this history. The case \(r=2\), namely \(\gamma(2)\le4/3\), is the classical Gr\"unbaum conjecture~\cite{Grunbaum60}, first proved by Chalmers and Lewicki~\cite{ChalmersLewicki10} and reproved by Basso~\cite{Basso19}. Low-dimensional cases beyond $r=2$ were first approached by more elaborate methods, for example in Chalmers and Lewicki~\cite{ChalmersLewicki09} and computationally in \cite{Basso19} and \cite{FoucartSkrzypek17}. Finally, Der\k{e}gowska and Lewandowska~\cite{DeregowskaLewandowska23}, inspired by work of Bukh and Cox~\cite{BukhCox20}, resolved the problem and gave a surprisingly short and elementary linear-algebraic proof of the general upper bound in \Cref{thm:projection-constant}.

In \Cref{sec:projection-constants}, we provide an alternative proof of this result using Gegenbauer polynomials by repairing the argument in \cite{KonigTomczak94}. While the key ideas of our proof are similar to \cite{DeregowskaLewandowska23}, we believe it is of independent interest. In particular, the slack terms in our argument are more transparent. This allows us to obtain a strict improvement on $\gamma(r)$ whenever $r$ is even and $r+2$ is not a perfect square. Appendix~\ref{app:improved-proj-bound} proves the following result.

\subsection*{\texorpdfstring{\Cref{thm:improved-proj-bound}}{The improved absolute projection constant bound}}
Let \(r\ge4\) be even, and suppose that \(r+2\) is not a perfect square. Then
\[
\gamma(r)\le
\frac{r+\sqrt{r+2}}{1+\sqrt{r+2}}
-
\frac{r}{2(2r)^{4r+8}}.
\]

\section{From projections to graph eigenvalues}\label{sec:reduction}
In this section, we prove our main theorem:
\begin{theorem}\label{thm:main-explicit}
Let $2 \leq k \leq n$.
For every graph $G$ on $n$ vertices,
\[
\lambda_k(G)\le \frac{\gamma(k-1)}{2(k-1)}\,n-1
\le \frac{(k-2)\sqrt{k+1}+2}{2k(k-1)}\, n-1.
\]
In particular, $c_k\le \frac{(k-2)\sqrt{k+1}+2}{2k(k-1)}$.
\end{theorem}

\begin{proof}
By the absolute projection constant bound in \Cref{thm:projection-constant},
$\gamma(r) \leq \frac{r+\sqrt{r+2}}{1+\sqrt{r+2}}$. Substituting $r=k-1$
gives $\frac{\gamma(k-1)}{2(k-1)} \leq
\frac{k-1 + \sqrt{k+1}}{2(k-1)(1+\sqrt{k+1})}$. Therefore, it suffices to
prove \Cref{thm:main-projection} below.
\end{proof}

\begin{theorem}\label{thm:main-projection}
Let \(2\le k\le n\), and let \(G\) be a graph on \(n\) vertices. Then
\[
\lambda_k(G)\le \frac{\gamma(k-1)}{2(k-1)}\,n-1.
\]
\end{theorem}

The key step in the proof of \Cref{thm:main-projection} is the following matrix
estimate for the sum of the \(r\) smallest eigenvalues.  Although we
will apply it to adjacency matrices, it holds for every symmetric matrix with
nonnegative diagonal entries and off-diagonal entries in \([0,1]\).

\begin{theorem}\label{thm:matrix-master}
Let $A=(a_{ij})\in \R^{n\times n}$ be symmetric, with eigenvalues
\[
\mu_1\ge \mu_2\ge \cdots\ge \mu_n.
\]
Assume that
\[
0\le a_{ij}\le 1 \quad (i\ne j),
\qquad
a_{ii}\ge 0 \quad (1\le i\le n).
\]
Let \(1\le r\le n\) be an integer. Then
\[
\mu_{n-r+1}+\cdots+\mu_n\ge -\frac{\gamma(r)}{2}\,n.
\]
In particular, $\mu_{n-r+1}\ge -\frac{\gamma(r)}{2r}\,n$.
\end{theorem}

\begin{proof}
By Ky Fan's minimum principle~\cite[Corollary~4.3.39]{HornJohnson12},
\[
\mu_{n-r+1}+\cdots+\mu_n=\min_{Q\in \cP_r(n)}\tr(AQ).
\]
Fix $Q=(q_{ij})\in \cP_r(n)$. Since $Q$ is positive semidefinite and $q_{ii} \geq 0$ for all $i \in \{1,\ldots,n\}$,
\begin{align*}
\tr(AQ)
&= \sum_{i,j} a_{ij} q_{ij}\\ 
&\geq \sum_{i,j} \min(0,q_{ij})\\ 
&= \frac{1}{2} \sum_{i,j} (q_{ij} - |q_{ij}|)\\ 
&= \frac{1}{2} \left( \1^\top Q \1 - \|Q\|_1\right)\\
&\geq -\frac{\|Q\|_1}{2}\\ 
&\geq -\frac{\gamma(r)}{2}\, n
\end{align*}
Taking the minimum over $Q\in \cP_r(n)$ gives the first claim. The second follows from
\[
r\,\mu_{n-r+1}\ge \mu_{n-r+1}+\cdots+\mu_n.
\qedhere\]
\end{proof}

\begin{proof}[Proof of \Cref{thm:main-projection}]
Let $r:=k-1$ and let $\overline G$ be the complement of $G$. Since
\[
A(G)+A(\overline G)=J-I,
\]
the eigenvalues of $J-I$ are $n-1,-1,\ldots,-1$, so in particular
\[
\lambda_2(J-I)=-1.
\]
Apply Weyl's inequality in the form
\[
\lambda_i(X)+\lambda_j(Y)\le \lambda_{i+j-n}(X+Y)
\qquad (i+j\ge n+1)
\]
with
\[
X=A(G),\qquad Y=A(\overline G),\qquad i=k,\qquad j=n-k+2.
\]
Since $i+j=n+2$, we obtain
\[
\lambda_k(G)+\lambda_{n-k+2}(\overline G)\le -1.
\]
Now apply \Cref{thm:matrix-master} with $r=k-1$ to $A(\overline G)$:
\[
\lambda_{n-k+2}(\overline G)\ge -\frac{\gamma(k-1)}{2(k-1)}\,n.
\]
Combining the two inequalities gives
\[
\lambda_k(G)\le \frac{\gamma(k-1)}{2(k-1)}\,n-1.
\qedhere\]
\end{proof}

\section{Positive kernels and Gegenbauer polynomials}

We need two positive-semidefinite kernels for the proof of
\Cref{thm:projection-constant}.

Fix an integer $r\ge 1$. Let $S^{r-1} := \{v \in \R^r: \|v\|_2 = 1\}$ be the unit sphere in $\R^r$, where $\|\cdot\|_2$ denotes the Euclidean norm. For real vectors $u,v$, let $\inner{u}{v}$ denote the standard inner product. For real-valued square matrices $A,B \in \R^{m\times m}$, write $\inner{A}{B}_F = \tr(A^{\top} B)$ for the Frobenius inner product. This is equivalent to taking the inner product after viewing $A$ and $B$ as $m^2$-dimensional vectors. For a matrix, $\|A\|_F = \sqrt{\froinner{A}{A}}$ and $\|A\|_1 = \sum_{i,j} |A_{ij}|$. Let $I_r \in \R^{r\times r}$ denote the identity matrix. A matrix $M \in \R^{m\times m}$ is positive semidefinite if $x^\top M x \geq 0$ for all $x \in \R^m$.

\begin{definition}
A function \(f:\R\to\R\) is positive semidefinite on \(S^{r-1}\) if, for every \(m\ge1\) and every \(u_1,\ldots,u_m\in S^{r-1}\), the matrix $\left(f(\inner{u_i}{u_j})\right)_{i,j=1}^m$ is positive semidefinite.
\end{definition}

Define 
\[
f^r_2(t) := t^2 - \frac{1}{r},
\qquad
f^r_4(t) := t^4 - \frac{3}{r(r+2)}.
\]
We prove that \(f_2^r\) and \(f_4^r\) are positive semidefinite on \(S^{r-1}\).
In \Cref{sec:projection-constants}, these kernels provide a scalar majorant for
\(\abs{\inner{u_i}{u_j}}\) in the estimate of \(\|Q\|_1\).

\begin{lemma}\label{lem:psd2}
For $r \geq 1$, define
\[
\phi(u):=uu^\top-\frac1r I_r,
\qquad u\in S^{r-1}.
\]
Then for all $u,v\in S^{r-1}$,
\[
\froinner{\phi(u)}{\phi(v)}=f_2^r(\inner{u}{v}).
\]
Further, $f_2^r$ is positive semidefinite on $S^{r-1}$.
\end{lemma}

\begin{proof}
A direct computation gives
\begin{align*}
\froinner{\phi(u)}{\phi(v)}
&=\tr\!\left(\left(uu^\top-\frac1r I_r\right)\left(vv^\top-\frac1r I_r\right)\right) \\
&=(u^\top v)^2-\frac1r
 =\inner{u}{v}^2-\frac1r
 =f_2^r(\inner{u}{v}).
\end{align*}
For any $u_1,\ldots,u_n \in S^{r-1}$, the matrix
\[
\bigl(f_2^r(\inner{u_i}{u_j})\bigr)_{i,j=1}^n
\]
is the Gram matrix of $\phi(u_1),\ldots,\phi(u_n)$ with respect to the
Frobenius inner product, and is therefore positive semidefinite.
\end{proof}

Appendix~\ref{app:psd4} gives the analogous Gram representation for \(f_4^r\)
using fourth-order tensors.

\begin{lemma}\label{lem:psd4}
For $r \geq 1$, \(f_4^r(t)\) is positive semidefinite on \(S^{r-1}\).
\end{lemma}

\begin{remark}
\Cref{lem:psd2} and \Cref{lem:psd4} are just special cases of a theorem by Schoenberg~\cite{Schoenberg42}. Let $G_\ell^{\lambda}$ denote the Gegenbauer polynomial of degree $\ell$. These polynomials have generating function~\cite{stein1971introduction}: 
\[
\frac{1}{(1 - 2tx + x^2)^\lambda} = \sum_{\ell=0}^\infty G_{\ell}^\lambda (t) x^{\ell}.
\]
For \(r\ge3\), Schoenberg proved the following strong characterization: $f(t)$ is a real continuous function such that the matrix $f(\inner{u_i}{u_j})_{i,j=1}^n$ is positive semidefinite for all subsets $\{u_1,\ldots,u_n\}\subseteq S^{r-1}$ if and only if $f$ is of the form $\sum_{\ell=0}^\infty a_{\ell} G_{\ell}^{r/2-1}$ with $a_{\ell} \geq 0$.

In particular, when \(r\ge3\), the normalized Gegenbauer kernels
\[
\widetilde G_\ell^{r/2-1}(t)
:=
\frac{G_\ell^{r/2-1}(t)}
     {G_\ell^{r/2-1}(1)}
\]
are positive semidefinite on \(S^{r-1}\). We have 
\[
f_2^r(t)=\frac{r-1}{r}\widetilde G_2^{\,r/2-1}(t) \quad \text{and} \quad f_4^r(t)=\frac{(r-1)(r+1)}{(r+2)(r+4)}\widetilde G_4^{\,r/2-1}(t)+\frac{6}{r+4}f_2^r(t).
\]
Thus \(f_2^r\) and \(f_4^r\) are positive linear combinations of normalized Gegenbauer kernels, and hence are positive semidefinite on \(S^{r-1}\) by \cite{Schoenberg42}. For \(r=2\), the same statement is the usual limiting Chebyshev, or Fourier-cosine, case for the circle. This case is also covered by Schoenberg's theorem, but not by the displayed generating function if one substitutes \(\lambda=0\) literally. In any case, the direct Gram proofs above cover all \(r\ge1\), so we do not require this strong characterization.  
\end{remark}

\section{A Gegenbauer-polynomial proof of the absolute projection constant bound}\label{sec:projection-constants}

In this section, we give an alternative proof of the upper bound on the absolute projection constant $\gamma(r)$ due to Der\k{e}gowska and Lewandowska~\cite{DeregowskaLewandowska23} by repairing an incorrect argument of \cite{KonigTomczak94} using Gegenbauer polynomials.

\begin{theorem}[\cite{DeregowskaLewandowska23}]\label{thm:projection-constant}
Let \(Q\in \R^{n\times n}\) be a rank-\(r\) orthogonal projection, where \(r\ge 1\). Then
\[
\|Q\|_1=\sum_{i,j=1}^n \abs{q_{ij}}
\le \frac{r+\sqrt{r+2}}{1+\sqrt{r+2}} n.
\]
In particular,
\[\gamma(r) \leq \frac{r+\sqrt{r+2}}{1+\sqrt{r+2}}.\]
\end{theorem}

\begin{proof}
Choose $V\in\R^{n\times r}$ with orthonormal columns such that $Q=VV^\top$.
The result is trivial for $r=1$, so assume $r\ge2$ for the rest of this
section. Let $x_1,\ldots,x_n\in\R^r$ be the rows of $V$. Write $x_i=c_i u_i$,
where $c_i=\norm{x_i}_2\ge0$ and $u_i\in S^{r-1}$ whenever $c_i\ne0$. If
$c_i=0$, choose $u_i$ arbitrarily in $S^{r-1}$. Since the columns of $V$ are
orthonormal, $V^\top V=I_r$. Writing this identity in terms of the rows gives
\begin{equation}\label{eq:projection-model}
\sum_{i=1}^n c_i^2 u_i u_i^\top=I_r 
\qquad\text{and}\qquad
q_{ij}=c_i c_j\inner{u_i}{u_j}
\end{equation}

Set
\[
C:=\sum_{i=1}^n c_i,
\qquad
x:=\sum_{i=1}^n c_i\phi(u_i),
\qquad
X:=\norm{x}_F.
\]

To prove \Cref{thm:projection-constant}, we establish \Cref{lem:CS-ineq,lem:abstract-majorant,lem:scalar-majorant}.

\begin{lemma}\label{lem:CS-ineq}
With the notation above,
\[
C^2+rX^2\le rn.
\]
\end{lemma}

\begin{proof}
\begin{align*}
C^2 + rX^2
&= \sum_{i=1}^n \sum_{j=1}^n c_i c_j + r \sum_{i=1}^n \sum_{j=1}^n c_i c_j \froinner{\phi(u_i)}{\phi(u_j)}\\
&= \sum_{i=1}^n \sum_{j=1}^n c_i c_j + r \sum_{i=1}^n \sum_{j=1}^n c_i c_j \left(\inner{u_i}{u_j}^2 - \frac{1}{r}\right) \text{\qquad by \Cref{lem:psd2}}\\
&= r \sum_{i=1}^n \sum_{j=1}^n c_i c_j \inner{u_i}{u_j}^2\\
&\leq \frac{r}{2}\sum_{i=1}^n \sum_{j=1}^n (c_i^2 + c_j^2) \inner{u_i}{u_j}^2 \text{\qquad by \(2c_i c_j\le c_i^2+c_j^2\)}\\
&= r \sum_{i=1}^n \sum_{j=1}^n c_j^2 \inner{u_i}{u_j}^2 \text{\qquad by symmetry}\\
&= r \sum_{i=1}^n \sum_{j=1}^n \tr(u_i u_i^\top c_j^2 u_j u_j^\top)\\
&= r \sum_{i=1}^n  \tr\left(u_i u_i^\top \sum_{j=1}^n c_j^2 u_j u_j^\top\right)\\
&= r \sum_{i=1}^n  \tr\left(u_i u_i^\top\right) \text{\qquad by \Cref{eq:projection-model}}\\
&= r n.
\qedhere\end{align*}
\end{proof}

We bound \(\|Q\|_1\) by finding a scalar majorant for \(|t|\) on
\([-1,1]\) and applying it with \(t=\inner{u_i}{u_j}\).  We consider
expressions of the form
\[
a+b\,f_2^r(t)-\rho\,f_4^r(t).
\]
After substituting \(t=\inner{u_i}{u_j}\) and summing over \(i,j\), the
constant term produces \(C^2\), the \(f_2^r\)-term produces \(X^2\), and the
sum involving \(f_4^r\) is nonnegative by \Cref{lem:psd4}.  Its coefficient
is \(-\rho\le0\), so dropping this term only increases the upper bound.  The
following lemma formalizes this reduction.

\begin{lemma}\label{lem:abstract-majorant}
Let \(a,b,\rho\in\R\) satisfy
\[
\abs{t}\le a+b\,f_2^r(t)-\rho\,f_4^r(t)
\qquad\text{for all } t\in[-1,1],
\]
with \(a\ge0\), \(\rho\ge 0\), and \(b\le ra\). Then every rank-\(r\) orthogonal projection \(Q\in\R^{n\times n}\) satisfies
\[
\|Q\|_1\le arn.
\]
\end{lemma}

\begin{proof}
By \eqref{eq:projection-model}, write $Q=VV^{\top}$ with rows $c_i u_i$ for
$i\in\{1,\ldots,n\}$. Then $\|Q\|_1=\sum_{i,j}|q_{ij}|=
\sum_{i,j}c_i c_j|\inner{u_i}{u_j}|$. Apply the assumed scalar inequality
with \(t=\inner{u_i}{u_j}\), multiply by $c_i c_j\ge0$, and sum over all
\(i,j\) to obtain
\[
\|Q\|_1 = \sum_{i,j} |q_{ij}|
\le a\sum_{i,j} c_i c_j
+b\sum_{i,j} c_i c_j f_2^r(\inner{u_i}{u_j})
-\rho\sum_{i,j} c_i c_j f_4^r(\inner{u_i}{u_j}).
\]
By definition,
\[
\sum_{i,j} c_i c_j=C^2,
\qquad
\sum_{i,j} c_i c_j f_2^r(\inner{u_i}{u_j}) = \sum_{i,j} c_i c_j \inner{\phi(u_i)}{\phi(u_j)}_F = X^2.
\]
Also, the positive semidefiniteness of \(\bigl(f_4^r(\inner{u_i}{u_j})\bigr)_{i,j=1}^n\) implies that
\[
\sum_{i,j} c_i c_j f_4^r(\inner{u_i}{u_j})\ge 0.
\]
Therefore $\|Q\|_1\le aC^2+bX^2$. Since \(b\le ra\), we have $aC^2+bX^2\le a\left(C^2+rX^2\right)$. Now apply \Cref{lem:CS-ineq} to get $\|Q\|_1\le arn$.
\end{proof}

The bound in \Cref{lem:abstract-majorant} is \(arn\), so we seek coefficients
that minimize \(ar\).  The equiangular tight cases suggest imposing equality at
\(t=1\) and \(t=1/\sqrt{r+2}\). \Cref{lem:scalar-majorant} gives the resulting
coefficients, and the following remark explains their derivation.

\begin{lemma}\label{lem:scalar-majorant}
Let \(r\ge 2\), and set $s:=\sqrt{r+2}$. Define
\[
a_r:=\frac{s^2+s-2}{r(s+1)},\qquad
b_r:=\frac{s(s^2+2s+3)}{2(s+1)^2},\qquad
\rho_r:=\frac{s^3}{2(s+1)^2}.
\]
Then for every \(t\in[-1,1]\),
\[
|t|\le a_r+b_r f_2^r(t)-\rho_r f_4^r(t).
\]
Moreover, we have $b_r\le ra_r$.
\end{lemma}

\begin{proof}
Recall that
\[
f_2^r(t)=t^2-\frac1r,\qquad
f_4^r(t)=t^4-\frac{3}{r(r+2)}.
\]
Since the left-hand side is even in \(t\), it suffices to consider \(t\in[0,1]\).
A direct expansion gives
\begin{equation}\label{eq:direct-expansion-of-hrt}
a_r+b_r f_2^r(t)-\rho_r f_4^r(t)-t
=
\frac{(1-t)(st-1)^2(st+s+2)}{2(s+1)^2}.
\end{equation}
The right-hand side is nonnegative for \(t\in[0,1]\), since
\[
1-t\ge 0,\qquad (st-1)^2\ge 0,\qquad st+s+2>0.
\]
Hence
\[
|t|\le a_r+b_r f_2^r(t)-\rho_r f_4^r(t)
\qquad\text{for all } t\in[-1,1].
\]

Finally,
\[
ra_r-b_r
=
\frac{s^3+2s^2-5s-4}{2(s+1)^2}
=
\frac{(s-2)(s+1)(s+3)+2}{2(s+1)^2}\ge 0,
\]
and therefore \(b_r\le ra_r\).
\end{proof}

\begin{remark}
The factorization in the proof also explains how these coefficients were found.
For the extremal equiangular constructions with
\[
N=\binom{r+1}{2},
\]
the common angle is
\[
\alpha=\frac1{\sqrt{r+2}}=\frac1s.
\]
Since we apply the polynomial majorant to quantities of the form
\[
|t|=|\inner{u_i}{u_j}|,
\]
it is natural to force contact at the endpoint \(t=1\) and at the equiangular
angle \(t=1/s\).  Thus one asks for
\[
H(1)=0,\qquad H(1/s)=0,\qquad H'(1/s)=0,
\]
where
\[
H(t):=a_r+b_r f_2^r(t)-\rho_r f_4^r(t)-t.
\]
These three conditions determine the coefficients above.
\end{remark}

\noindent\textit{Completing the proof of \Cref{thm:projection-constant}}. Apply \Cref{lem:abstract-majorant} with
\[
a=a_r,\qquad b=b_r,\qquad \rho=\rho_r.
\]
\Cref{lem:scalar-majorant} verifies the hypotheses of
\Cref{lem:abstract-majorant}. Therefore
\[
\|Q\|_1\le a_r r n = \frac{s^2 + s -2}{s+1}\, n = \frac{r+\sqrt{r+2}}{1+\sqrt{r+2}}\, n.
\qedhere\]
\end{proof}

The absolute projection constant bound in \Cref{thm:projection-constant} is sharp in dimensions $r\in\{1,2,3,7,23\}$ because extremal real equiangular tight frames are known to exist in these dimensions. Refer to \Cref{sec:weighted-equiangular} for further discussion.

The same slack identity gives a strict improvement in the nonsquare dimensions.
We defer the proof to Appendix~\ref{app:improved-proj-bound}.

\begin{theorem}\label{thm:improved-proj-bound}
Let \(r\ge4\) be even, and suppose that \(r+2\) is not a perfect square. Then,
for every rank-\(r\) orthogonal projection \(Q\in\R^{n\times n}\),
\[
    \|Q\|_1\le
    \left(
    \frac{r+\sqrt{r+2}}{1+\sqrt{r+2}}
    -
    \frac{r}{2(2r)^{4r+8}}
    \right)n.
\]
Consequently,
\[
\gamma(r)\le
\frac{r+\sqrt{r+2}}{1+\sqrt{r+2}}
-
\frac{r}{2(2r)^{4r+8}}.
\]
\end{theorem}

\section{Lower bounds from weighted two-graph blowups}\label{sec:weighted-lower}

\subsection{Equiangular tight cases}\label{sec:weighted-equiangular}

Suppose there is an extremal equiangular line system in \(\R^r\), so that
\[
N=\binom{r+1}{2}
\]
lines exist.  Choose unit representatives \(u_1,\ldots,u_N\) with
\[
\inner{u_i}{u_j}\in\{\pm \alpha\}\qquad (i\ne j).
\]
It is standard that equality in the relative bound gives an equiangular tight
frame. See, for example, \cite{DGS77,lemmens1973equiangular}.  In particular,
if \(U\in\R^{r\times N}\) has columns \(u_1,\ldots,u_N\), then
\[
UU^\top=\frac Nr I_r,
\]
and extremality forces
\[
\alpha=\frac1{\sqrt{r+2}}.
\]
The corresponding sign matrix is
\[
B_{ii}=1,\qquad
B_{ij}=\operatorname{sign}\inner{u_i}{u_j}\quad (i\ne j).
\]
The Gram matrix \(G:=U^\top U\) therefore has entries
\[
G_{ii}=1,\qquad G_{ij}=\alpha B_{ij}\quad (i\ne j),
\]
or equivalently
\[
G=I_N+\alpha(B-I_N)=\alpha B+(1-\alpha)I_N.
\]
Since \(UU^\top=(N/r)I_r\), the nonzero eigenvalues of \(G=U^\top U\) are $N/r$ with multiplicity \(r\), and the remaining eigenvalues are \(0\).  It follows
that \(B\) has eigenvalue
\[
\theta
:=
\frac{\frac Nr-1+\alpha}{\alpha}
=
\frac{N\beta_r}{r},
\qquad
\beta_r:=\frac{r+\sqrt{r+2}}{1+\sqrt{r+2}},
\]
with multiplicity \(r\), while the remaining eigenvalues are $-(1-\alpha)/\alpha$.

In these cases, no weight optimization is needed.  Let \(Q\) be the orthogonal
projector onto the \(N\beta_r/r\)-eigenspace of \(B\).  Since $G=\alpha B+(1-\alpha)I_N$ has range equal to this eigenspace, we have
\[
Q=\frac rN\,G.
\]
Thus \(Q\) is a rank-\(r\) orthogonal projection and, entrywise,
\[
Q_{ii}=\frac rN>0,\qquad
Q_{ij}=\frac{r\alpha}{N}B_{ij}\quad(i\ne j).
\]
Hence \(\operatorname{sign}(Q)=B\).  Therefore
\[
\frac{\|Q\|_1}{N}
=
\frac{\operatorname{tr}(BQ)}{N}
=
\frac1N\sum_{j=1}^r \lambda_j(B)
=
\beta_r.
\]
Thus the equiangular sign matrix recovers the sharp absolute projection
constant in the dimensions where maximal real equiangular tight frames are
known to exist, namely $r\in\{1,2,3,7,23\}$.

The lines also describe the associated graph directly. Its vertex set is
\[
\{i^+,i^-:1\le i\le N\}.
\]
For \(i\ne j\), if \(\inner{u_i}{u_j}=\alpha\), join
\[
i^+j^+,\qquad i^-j^-,
\]
and if \(\inner{u_i}{u_j}=-\alpha\), join
\[
i^+j^-,\qquad i^-j^+.
\]
This is the regular two-graph construction used in the equiangular-line
examples \cite{Seidel1976SurveyTwoGraphs}.  Let \(G_0\) denote this base graph
on \(2N\) vertices.  The spectrum of this graph is well understood
\cite{Seidel1976SurveyTwoGraphs,GodsilHensel1992}. In particular, its adjacency
matrix satisfies the exact
identity
\[
\lambda_{r+1}(G_0)
=
\frac{\beta_r}{2r}\,|V(G_0)|-1.
\]
Moreover, replacing every vertex of \(G_0\) by a clique of the same size and
every edge by a complete bipartite graph preserves the identity:
\[
\lambda_{r+1}(G)
=
\frac{\beta_r}{2r}\,|V(G)|-1.
\]
Thus these graphs match the upper bound in \Cref{thm:main-explicit} exactly
when \(k=r+1\).
These are precisely the tight cases \(k\in\{2,3,4,8,24\}\) coming from
equiangular lines, in agreement with Linz's lower-bound constructions
\cite{Linz23}.

The extremal equiangular line systems in dimensions \(r=7\) and \(r=23\) are
derived from the \(E_8\) and Leech lattices, respectively
\cite{gillespie2018equiangular}.  These lattices also give the sharp sphere
packings in dimensions \(8\) and \(24\)
\cite{viazovska2017sphere,cohn2017sphere}.

\subsection{Finite weighted cores and blowups}

The equal-weight sign matrix in the preceding subsection suggests the following
generalization.  A \emph{finite weighted core} is a pair \((B,p)\), where
\[
B\in\{\pm1\}^{m\times m},
\qquad B=B^\top,\qquad B_{ii}=1,
\]
and \(p=(p_1,\ldots,p_m)\) satisfies \(p_i\ge0\) and
\(\sum_i p_i=1\).  The indices are called \emph{types}: \(B_{ij}\) prescribes
the sign between types \(i\) and \(j\), and \(p_i\) is the asymptotic
proportion of indices of type \(i\).  For integers \(N_i\) with
\(N=\sum_iN_i\) and \(N_i/N\to p_i\), replacing \(B_{ij}\) by an
\(N_i\times N_j\) constant block produces an \(N\times N\) sign matrix
\(B^{(N)}\), called a \emph{weighted blowup} of the core.  The number \(m\)
remains fixed as \(N\) grows.

The terminology comes from writing a rank-\(r\) projection as \(Q=XX^\top\).
When \(q_{ab}\ne0\), the directions of the \(a\)-th and \(b\)-th rows of
\(X\) determine its sign.  Rows with the same direction up to sign form a
type, and the signs between chosen representatives give the matrix \(B\).
Replacing one representative by its negative multiplies the corresponding row
and column of \(B\) by \(-1\). Replacing several representatives changes \(B\)
to \(\Sigma B\Sigma\), where \(\Sigma\) is diagonal with entries in
\(\{\pm1\}\).  This operation is called \emph{switching}, and a switching class
is a two-graph.  This framework is essentially due to Basso, who formulated
the absolute projection constant through eigenvalue problems for weighted
finite two-graphs \cite{Basso19}.

For a finite weighted core, put
\[
D_p:=\operatorname{diag}(\sqrt{p_1},\ldots,\sqrt{p_m}),
\qquad
K_p:=D_pBD_p.
\]
We may delete coordinates with \(p_i=0\), so we assume below that all weights
are positive.  Define the projection-side objective
\[
\Phi_B(p):=\sum_{j=1}^r\lambda_j(K_p).
\]

\begin{proposition}\label{prop:weighted-core-blowups}
Let \((B,p)\) be a finite weighted core with positive weights, \(m\ge r\), and
\(\lambda_r(K_p)>0\).  Then
\[
\gamma(r)\ge \Phi_B(p).
\]
If also \(\lambda_1(K_p)<1\), then the same core gives
\[
c_{r+1}\ge \frac12\lambda_r(K_p).
\]
\end{proposition}

\begin{lemma}[Constant-block blowups]\label{lem:constant-block-spectrum}
Let \(C\) be a symmetric \(s\times s\) matrix, and let
\(\mathbf n=(n_1,\ldots,n_s)\) be a vector of positive integers with
\(n=\sum_i n_i\).  Form \(C[\mathbf n]\) by replacing the entry \(C_{ij}\)
with an \(n_i\times n_j\) constant block.  Put \(q_i=n_i/n\) and
\(D_q:=\operatorname{diag}(\sqrt{q_1},\ldots,\sqrt{q_s})\).  Then the spectrum
of \(n^{-1}C[\mathbf n]\) consists of the eigenvalues of
\(D_qCD_q\), together with \(n-s\) additional zeros.
\end{lemma}

\begin{proof}
Partition the indices of \(C[\mathbf n]\) into blocks
\(V_1,\ldots,V_s\), where \(|V_i|=n_i\), and define
\(Z\in\{0,1\}^{n\times s}\) by \(Z_{ai}=1\) if \(a\in V_i\) and \(Z_{ai}=0\)
otherwise.  Thus each row of \(Z\) indicates the type of one blown-up index.
This is the type-incidence matrix.  If
\(D_{\mathbf n}:=\operatorname{diag}(\sqrt{n_1},\ldots,\sqrt{n_s})\), then
\(C[\mathbf n]=ZCZ^\top\) and \(Z^\top Z=D_{\mathbf n}^2\).
Consequently,
\(U:=ZD_{\mathbf n}^{-1}\) has orthonormal columns.  Since
\(D_{\mathbf n}/\sqrt n=D_q\),
\[
\frac1n C[\mathbf n]
=
U\left(\frac{D_{\mathbf n}}{\sqrt n}
       C
       \frac{D_{\mathbf n}}{\sqrt n}\right)U^\top
=U(D_qCD_q)U^\top.
\]
Set \(L:=D_qCD_q\).  Since the columns of \(U\) are orthonormal, every
vector in \(\mathbb R^n\) can be written uniquely as \(Uy+x\), where
\(y\in\mathbb R^s\) and \(x\) is perpendicular to every column of \(U\).
The identities \(U^\top U=I_s\) and \(U^\top x=0\) give
\(n^{-1}C[\mathbf n](Uy+x)=ULU^\top(Uy+x)=ULy\).  In particular, if
\(Ly=\lambda y\), then \(Uy\) is an eigenvector of \(n^{-1}C[\mathbf n]\)
with eigenvalue \(\lambda\), while every nonzero vector perpendicular to the
columns of \(U\) is an eigenvector with eigenvalue zero.  These two subspaces
have dimensions \(s\) and \(n-s\), so these are all the eigenvalues claimed
in the lemma.
\end{proof}

\begin{proof}[Proof of \Cref{prop:weighted-core-blowups}]
\emph{Projection bound.}
Choose positive integers \(N_1,\ldots,N_m\) with
\(N:=\sum_iN_i\to\infty\) and \(N_i/N\to p_i\), and let \(B^{(N)}\) be the
corresponding weighted blowup.  Write
\(q^{(N)}=(N_1/N,\ldots,N_m/N)\).  By
\Cref{lem:constant-block-spectrum}, the spectrum of \(N^{-1}B^{(N)}\) consists
of the eigenvalues of \(D_{q^{(N)}}BD_{q^{(N)}}\), together with additional
zeros.  The former matrix converges to \(K_p\).

Let \(Q_N\) be the orthogonal projector onto a top \(r\)-dimensional
eigenspace of \(B^{(N)}\).  The assumption \(\lambda_r(K_p)>0\) ensures that,
for all sufficiently large \(N\), these \(r\) eigenvalues do not come from the
additional zeros in the lemma.  Hence Ky Fan's principle and continuity of the
ordered eigenvalues give
\[
\frac1N\operatorname{tr}(B^{(N)}Q_N)
=
\sum_{j=1}^r
\lambda_j\left(D_{q^{(N)}}BD_{q^{(N)}}\right)
\longrightarrow
\Phi_B(p).
\]
Since every entry of \(B^{(N)}\) is a sign,
\[
\|Q_N\|_1\ge \operatorname{tr}(B^{(N)}Q_N).
\]
It follows that \(\gamma(r)\ge\Phi_B(p)\).  The last comparison is an equality
whenever \(\operatorname{sign}(Q_N)=B^{(N)}\).

\emph{Graph bound.}
Recall that \(D_p=\operatorname{diag}(\sqrt{p_1},\ldots,\sqrt{p_m})\) and
\(K_p=D_pBD_p\).
Let \(J\) be the \(m\times m\) all-ones matrix.  Index rows and columns by
\(\{i^+,i^-:1\le i\le m\}\), and define
\[
C_B:=
\frac12
\begin{pmatrix}
J+B & J-B\\
J-B & J+B
\end{pmatrix}.
\]
Since \(B_{ii}=1\), we have \(C_B=A+I_{2m}\), where \(A\) is the adjacency
matrix of the simple graph that joins equal signs when \(B_{ij}=1\), joins
opposite signs when \(B_{ij}=-1\), and has no edge \(i^+i^-\).

Choose integers \(t_i=t_i(T)>0\) with \(\sum_i t_i=T\) and \(t_i/T\to p_i\).
Replace \(i^\pm\) by a clique of size \(t_i\), joining two replacement cliques
completely when their base vertices are adjacent.  Let
\(G_M\) be the resulting graph, where \(M=2T\).
Adding \(I_M\) turns each diagonal clique block into an all-ones block.
Thus every block of \(A(G_M)+I_M\) is constant, and its value is the
corresponding entry of \(C_B\).  Applying \Cref{lem:constant-block-spectrum}
to \(C_B\) with these \(2m\) block sizes uses
\(q^{(T)}=M^{-1}(t_1,\ldots,t_m,t_1,\ldots,t_m)\).  Each of its two coordinates
for type \(i\) is \(t_i/M=(t_i/T)/2\to p_i/2\).  Thus the eigenvalues of
\(M^{-1}(A(G_M)+I_M)\) are those of
\(D_{q^{(T)}}C_BD_{q^{(T)}}\), together with \(M-2m\) additional zeros, and the
smaller matrix converges entry by entry to \(\widehat A(p):=H_pC_BH_p\), where
\(H_p:=\operatorname{diag}(D_p/\sqrt2,D_p/\sqrt2)\).
For symmetric matrices of fixed size, Weyl's eigenvalue perturbation inequality
shows that entry-by-entry convergence implies convergence of the ordered
eigenvalues~\cite[Section~4.3]{HornJohnson12}.

Write a vector in \(\mathbb R^{2m}\) as \((u,v)\), where
\(u,v\in\mathbb R^m\).  Setting \(x=(u+v)/2\) and \(z=(u-v)/2\) gives
\((u,v)=(x,x)+(z,-z)\).  The displayed formula for \(C_B\) now gives
\[
\widehat A(p)(x,x)
=\left(\frac12D_pJD_px,\frac12D_pJD_px\right)
\quad\text{and}\quad
\widehat A(p)(z,-z)
=\left(\frac12K_pz,-\frac12K_pz\right).
\]
The matrices \(D_pJD_p/2\) and \(K_p/2\) are symmetric, so choose eigenvector
bases \(x_1,\ldots,x_m\) and \(z_1,\ldots,z_m\) for them.  The calculation
above shows that \((x_i,x_i)\) and \((z_i,-z_i)\), for \(1\le i\le m\), are
eigenvectors of \(\widehat A(p)\) with the corresponding eigenvalues.  These
\(2m\) vectors are linearly independent: any relation among them has the form
\((x,x)+(z,-z)=0\), which gives \(x+z=x-z=0\), hence \(x=z=0\), and then all
coefficients vanish because the \(x_i\) and \(z_i\) are bases.  They therefore
form a basis of \(\mathbb R^{2m}\), so every eigenvalue of \(\widehat A(p)\)
comes from one of these two \(m\times m\) matrices.

Let \(w=(\sqrt{p_1},\ldots,\sqrt{p_m})^\top\).  Since \(J\) is the all-ones
matrix, \(D_pJD_p=ww^\top\), and \(w^\top w=\sum_i p_i=1\).  Thus
\(D_pJD_p/2\) sends \(w\) to \(w/2\) and sends every vector perpendicular to
\(w\) to zero.  Its eigenvalues are therefore \(1/2\) and \(m-1\) zeros, while
\(K_p/2\) has eigenvalues \(\lambda_j(K_p)/2\).  Because
\(\lambda_1(K_p)<1\) and
\(\lambda_r(K_p)>0\), the largest \(r+1\) eigenvalues of \(\widehat A(p)\) are
\(1/2\) and \(\lambda_1(K_p)/2,\ldots,\lambda_r(K_p)/2\).  Hence
\(\lambda_{r+1}(\widehat A(p))=\lambda_r(K_p)/2>0\).  The additional zeros from
\Cref{lem:constant-block-spectrum} lie below the largest \(r+1\) eigenvalues
for all sufficiently large \(M\).  Consequently,
\[
\frac{\lambda_{r+1}(A(G_M)+I_M)}{M}
\longrightarrow
\frac12\lambda_r(K_p).
\]
Finally,
\(\lambda_{r+1}(A(G_M)+I_M)=\lambda_{r+1}(G_M)+1\).  Since
\(M=|V(G_M)|\to\infty\), dividing by \(M\) gives
\[
\frac{\lambda_{r+1}(G_M)}{|V(G_M)|}
\longrightarrow
\frac12\lambda_r(K_p),
\]
and hence \(c_{r+1}\ge\lambda_r(K_p)/2\).
\end{proof}

\subsection{\texorpdfstring{Improved lower bounds for \(c_k\) for small \(k\)}{Improved lower bounds for ck for small k}}

We use deterministic local searches to produce candidate cores for
\Cref{prop:weighted-core-blowups}, optimizing \(\Phi_B(p)\) and
\(\lambda_r(K_p)\) separately.  We check each listed construction directly.
The searches are heuristic: we make no claim of global optimality,
stabilization, minimality, or a uniform bound on the core size.

\Cref{tab:computed-weighted-cores} summarizes the best cores found.  The
projection and graph columns may use different sign matrices and weights.
The quantities \(m_P\) and \(m_G\) denote the respective numbers of positive weights.  The
entries for \(r=4\) and \(r=5\) come from the explicit constructions described
in Appendix~\ref{app:weighted-examples}.  Apart from the tight equiangular
entries already explained, finite-type extraction produced the remaining
numerical weighted-core evidence.  The ancillary
files \texttt{anc/weighted\_cores.json} and
\texttt{anc/verify\_weighted\_cores.py} contain the objective-specific sign
cores and weights and a direct verification script for all rows.  The same
directory also contains the search code.
The table gives the lower bound for \(\gamma(r)\) obtained from the core,
the scaled projection value \(\gamma(r)/(2r)\), the resulting lower bound for
\(c_k\), and the gap between these last two quantities.

\begin{table}[htbp]
\centering
\scriptsize
\setlength{\tabcolsep}{3pt}
\begin{tabular}{c c c c c c c c}
\hline
\(r\) & \(k\) & \(m_P\) & \(m_G\) & \(\gamma(r)\) & \(\gamma(r)/(2r)\) &
\(c_k\) & gap \\
\hline
4 & 5 & 11 & 11 & 1.8500 & 0.2312 & 0.2312 & \(5.51{\cdot}10^{-5}\)\\
5 & 6 & 16 & 16 & 2.0691 & 0.2069 & 0.2069 & 0\\
6 & 7 & 21 & 21 & 2.2857 & 0.1904 & 0.1904 & 0\\
7 & 8 & 28 & 28 & 2.5000 & 0.1785 & 0.1785 & 0\\
8 & 9 & 36 & 36 & 2.6666 & 0.1666 & 0.1666 & 0\\
9 & 10 & 46 & 46 & 2.8202 & 0.1566 & 0.1566 & \(3.69{\cdot}10^{-5}\)\\
10 & 11 & 57 & 57 & 2.9499 & 0.1474 & 0.1473 & \(1.29{\cdot}10^{-4}\)\\
11 & 12 & 68 & 71 & 3.0537 & 0.1388 & 0.1386 & \(1.61{\cdot}10^{-4}\)\\
12 & 13 & 83 & 85 & 3.1687 & 0.1320 & 0.1319 & \(1.28{\cdot}10^{-4}\)\\
13 & 14 & 95 & 105 & 3.2868 & 0.1264 & 0.1262 & \(1.51{\cdot}10^{-4}\)\\
14 & 15 & 114 & 119 & 3.3981 & 0.1213 & 0.1212 & \(1.43{\cdot}10^{-4}\)\\
15 & 16 & 130 & 153 & 3.5106 & 0.1170 & 0.1168 & \(1.59{\cdot}10^{-4}\)\\
16 & 17 & 145 & 152 & 3.6157 & 0.1129 & 0.1128 & \(1.60{\cdot}10^{-4}\)\\
\hline
\end{tabular}
\caption{The indicated independently optimized weighted-core constructions
supply all displayed lower bounds in the \(\gamma(r)\), \(\gamma(r)/(2r)\),
and \(c_k\) columns, with \(k=r+1\).  Here \(m_P\) and
\(m_G\) are the numbers of positive weights in the projection and graph cores.
We truncate each numerical lower bound downward to four decimal places and
compute the gap from the unrounded construction values.}
\label{tab:computed-weighted-cores}
\end{table}

Comparing with Linz's table of lower bounds \cite[Table~1]{Linz23}, these
weighted-core constructions match the values for \(k=7,8,9\), while improving
the known lower bounds for
\[
k\in\{5,6,10,11,12,13,14,15,16,17\}.
\]
In particular, Linz proved \(c_5\ge2/9\) using the Paley graph and
\(c_6\ge0.2\). The constructions in Appendix~\ref{app:weighted-examples} improve both
of these bounds.
For \(r=4\) and \(9\le r\le16\), the computed scaled projection lower bound is
strictly larger than the computed graph lower bound.  Since we optimized the
two objectives independently, this does not imply that the underlying extremal
constants differ.

We outline the first two concrete examples in Appendix~\ref{app:weighted-examples}.

\FloatBarrier
\section{Discussion}\label{sec:discussion}

\subsection{A possible gap between the constants}

Our reduction gives
\[
    c_{r+1}\le \frac{\gamma(r)}{2r}.
\]
For several values of \(r\), the numerical experiments summarized in
\Cref{tab:computed-weighted-cores} produce slightly different best lower bounds
for the two sides, suggesting that this inequality may be strict.  Since the
heuristic local searches produced the constructions and the observed gaps are
small, this does not establish a separation.  It would be interesting to
determine whether \(c_{r+1}<\gamma(r)/(2r)\) for some \(r\), or whether equality
always holds and the gaps in the table are numerical artifacts.

\subsection{Sums of the largest eigenvalues and graph energy}

Nikiforov also discusses a closely related problem involving sums of the
largest eigenvalues \cite[Section~4]{Nikiforov15}.  For a graph \(G\), write
\[
\Lambda_k(G):=\lambda_1(G)+\cdots+\lambda_k(G),
\]
and let
\[
M_k(n):=\max\{\Lambda_k(G): |V(G)|=n\}.
\]
This problem is closely related to the study of Ky Fan norms and graph energy,
where one instead sums the largest singular values. Mohar proved the general upper bound
\[
M_k(n)\le \frac{1+\sqrt{k}}2\,n
\]
and asked how close it is to the truth \cite{Mohar09}. See also
\cite[Section~4]{Nikiforov15}. Nikiforov formulated the following concrete
conjecture~\cite[Conjecture 4.2]{Nikiforov15}: for every fixed \(k\ge2\), if
$\tau_k:=\lim_{n\to\infty}\frac{M_k(n)}n$, then there exists
\(\varepsilon_k>0\) such that
$\tau_k<\frac12\left(1+\sqrt{k}-\varepsilon_k\right)$. The following proposition
resolves this conjecture and proves the stronger finite-\(n\) statement.

\begin{proposition}\label{prop:spectral-sum}
For every graph \(G\) on \(n\) vertices and every \(1\le k\le n\),
\[
\Lambda_k(G)
\le
\frac{1+\gamma(k)}2\,n-k.
\]
Consequently,
\[
M_k(n)
\le
\frac12\left(1+\frac{k+\sqrt{k+2}}{1+\sqrt{k+2}}\right)n-k.
\]
\end{proposition}

\begin{proof}
Let \(A\) be the adjacency matrix of \(G\).  By Ky Fan's maximum principle,
\[
\Lambda_k(G)=\max_{Q\in\cP_k(n)}\tr(AQ).
\]
Fix \(Q=(q_{ij})\in\cP_k(n)\).  Since \(a_{ii}=0\) and \(0\le a_{ij}\le1\)
for \(i\ne j\),
\[
\tr(AQ)
=\sum_{i\ne j}a_{ij}q_{ij}
\le
\sum_{i\ne j}\max(0,q_{ij}).
\]
Using \(\max(0,x)=(x+|x|)/2\), this becomes
\[
\tr(AQ)
\le
\frac12\left(
\sum_{i\ne j}q_{ij}+\sum_{i\ne j}|q_{ij}|
\right)
=
\frac12\left(\1^\top Q\1+\|Q\|_1-2k\right).
\]
Now \(\1^\top Q\1=\|Q\1\|_2^2\le \|\1\|_2^2=n\), while
\(\|Q\|_1\le \gamma(k)n\).  Hence
\[
\tr(AQ)\le \frac{1+\gamma(k)}2\,n-k.
\]
Taking the maximum over \(Q\in\cP_k(n)\) proves the first claim.  The second
claim follows from \Cref{thm:projection-constant}.
\end{proof}

As an immediate consequence, \Cref{prop:spectral-sum} resolves
\cite[Conjecture~4.2]{Nikiforov15}.  Indeed, for every \(k\ge2\), $(k+\sqrt{k+2})/(1+\sqrt{k+2})<\sqrt{k}$, and hence
\[
\tau_k
=
\lim_{n\to\infty}\frac{M_k(n)}n
\le
\frac12\left(1+\frac{k+\sqrt{k+2}}{1+\sqrt{k+2}}\right)
<
\frac12(1+\sqrt{k}).
\]
Thus the conjecture holds with
\[
\varepsilon_k
=
\sqrt{k}
-
\frac{k+\sqrt{k+2}}{1+\sqrt{k+2}}
>0.
\]

This argument also identifies the mixed projection constant
\[
\eta_{\mathrm{mix}}(k):=\sup_{N\ge k}\frac1N
\max_{Q\in\cP_k(N)}\bigl(\|Q\|_1+\1^\top Q\1\bigr),
\]
which is already implicit in the \(k=2\) spectral-sum literature.  The same
proof gives the one-line bound
\[
\Lambda_k(G)\le \frac{\eta_{\mathrm{mix}}(k)}2\,n-k
\qquad\text{with}\qquad
\eta_{\mathrm{mix}}(k)\le 1+\gamma(k).
\]

For \(k=2\), specialized work gives a much sharper answer.  Ebrahimi, Mohar,
Nikiforov, and Ahmady proved
\[
\lambda_1(G)+\lambda_2(G)<\frac{8.0185}{7}\,n
\]
and conjectured the sharp bound \(8n/7\) \cite{EbrahimiMoharNikiforovAhmady08}.
Recently, Kumar, Liu, Monterde, Pragada, and Tait proved this conjecture
using graphon and convex-analytic methods~\cite{KumarLiuMonterdePragadaTait26}:
\[
\lambda_1(G)+\lambda_2(G)\le \frac87\,n.
\]
Agarwal et al. also proved the related Nordhaus--Gaddum
analogue~\cite{AgarwalEtAl26}:
\[
\max_{|V(G)|=n}\bigl(\lambda_1(G)+\lambda_2(\overline G)\bigr)
 =\left(\frac87+o(1)\right)n.
\]
Our general estimate gives only
\[
\lambda_1(G)+\lambda_2(G)\le \frac76\,n-2.
\]
For every \(k>2\), however, the same projection-constant argument improves the
general \((1+\sqrt{k})n/2\) bound in \cite{Mohar09}.

\subsection{Multiplicity of the second eigenvalue}

Assume that \(G\) is connected, and let $m$ be the multiplicity of
\(\lambda_2(G)\). Then
\[
\lambda_2(G)=\lambda_{m+1}(G)\le \alpha_{m+1}n-1.
\]
Consequently,
\[
m\le \max\bigl\{1\le i\le n-1:\ \lambda_2(G)+1\le \alpha_{i+1}n\bigr\}.
\]
Thus \Cref{thm:main-explicit} gives an implicit upper bound on the multiplicity of the second eigenvalue in terms of $\lambda_2(G)$ and $n$.

If \(G=K_n\), then \(\lambda_2(G)=-1\) and \(m=n-1\), so the displayed finite
maximum gives only the trivial bound.  For connected non-complete graphs, we
have \(\lambda_2(G)+1>0\), and since \(1/\alpha_{i+1}=\Theta(\sqrt{i})\), the
same inequality yields
\[
m=O\left(\left(\frac{n}{\lambda_2(G)+1}\right)^2\right).
\]
The asymptotic order already follows from Nikiforov's bound~\cite[Theorem~2.6]{Nikiforov15}, but our estimate is sharper.

\subsection{Sparse graphs}

Our method is tailored to dense graphs, since it treats the complement and uses
only the coarse information $0\le a_{ij}\le 1$.  For sparse graphs, the trace
identity $\sum_{i=1}^n \lambda_i(G)^2=2e(G)$ implies
\[
\lambda_k(G)\le \sqrt{\frac{2e(G)}{k}}.
\]
An open problem is to obtain bounds that interpolate effectively between the
dense regime of \Cref{thm:main-explicit} and sparse or bounded-degree regimes.

\bigskip
\noindent\textbf{Acknowledgments.}
We thank Noga Alon, Matija Buci\'c, Sergio Cristancho, Julien Codsi, Alex Divoux, Clive Elphick, Andrew Lin, Jie Ma, Davin Park, and Shouda Wang for helpful discussions. This paper was written with assistance from GPT-5.4. All theorem statements, proofs, computations, references, and final arguments were independently checked, revised, and finalized by the authors, who take full responsibility for the correctness, originality, and integrity of the paper. The first author was supported by the National Science Foundation under Grant No. DMS-2349013.
% \pagebreak

\appendix
\section{Positivity of the degree-four kernel}\label{app:psd4}

We prove \Cref{lem:psd4} using $4$-tensors. A $4$-tensor
$T=(T_{abcd})_{a,b,c,d=1}^r$ is an array of real numbers with four indices.  
If $T=(T_{abcd})$ and $S=(S_{abcd})$ are two $4$-tensors of the same size, their Frobenius inner product is defined by
\[
\froinner{T}{S}:=\sum_{a,b,c,d=1}^r T_{abcd}S_{abcd}.
\]
This is the natural analogue of the usual Frobenius inner product for matrices.  
For a vector $u=(u_1,\dots,u_r)\in\mathbb R^r$, we write $u^{\otimes 4}$ for the $4$-tensor with entries
\[
(u^{\otimes 4})_{ijkl}:=u_i u_j u_k u_l.
\]
A $4$-tensor is called symmetric if its entries do not change when the indices are permuted.

\subsection*{\texorpdfstring{\Cref{lem:psd4}}{The degree-four kernel}} For $r \geq 1$,
\[
f_4^r(t)=t^4-\frac{3}{r(r+2)}
\]
is positive semidefinite on \(S^{r-1}\).

\begin{proof}
Let \(\Omega=(\Omega_{ijkl})_{1\le i,j,k,l\le r}\) be the symmetric \(4\)-tensor
\[
\Omega_{ijkl}:=\frac{1}{r(r+2)}
\bigl(\delta_{ij}\delta_{kl}+\delta_{ik}\delta_{jl}+\delta_{il}\delta_{jk}\bigr),
\]
where \(\delta_{ab}\) is the Kronecker delta, that is,
\[
\delta_{ab}=
\begin{cases}
1,& a=b,\\
0,& a\neq b.
\end{cases}
\]
For \(u\in S^{r-1}\), define
\[
\psi(u):=u^{\otimes 4}-\Omega.
\]
We claim that for all \(u,v\in S^{r-1}\),
\[
\froinner{\psi(u)}{\psi(v)}=f_4^r(\inner{u}{v}).
\]
Write \(t=\inner{u}{v}\). We compute each term separately.

First,
\[
\froinner{u^{\otimes 4}}{v^{\otimes 4}}
=\sum_{i,j,k,l} u_i u_j u_k u_l\, v_i v_j v_k v_l
=(u^\top v)^4=t^4.
\]

Next,
\begin{align*}
\froinner{u^{\otimes 4}}{\Omega}
&=\frac{1}{r(r+2)}
\sum_{i,j,k,l}u_i u_j u_k u_l
\bigl(\delta_{ij}\delta_{kl}+\delta_{ik}\delta_{jl}+\delta_{il}\delta_{jk}\bigr).
\end{align*}
Each of the three terms contributes
\[
\sum_{i,j,k,l}u_i u_j u_k u_l\,\delta_{ij}\delta_{kl}
=\sum_{i,k}u_i^2u_k^2
=\Bigl(\sum_i u_i^2\Bigr)^2
=1,
\]
so
\[
\froinner{u^{\otimes 4}}{\Omega}=\frac{3}{r(r+2)}.
\]
Similarly,
\[
\froinner{v^{\otimes 4}}{\Omega}=\frac{3}{r(r+2)}.
\]

It remains to compute \(\froinner{\Omega}{\Omega}\). Let
\[
C_{ijkl}:=\delta_{ij}\delta_{kl}+\delta_{ik}\delta_{jl}+\delta_{il}\delta_{jk},
\]
so that \(\Omega=\frac{1}{r(r+2)}C\). Then
\[
\froinner{C}{C}
=\sum_{i,j,k,l} C_{ijkl}^2.
\]
Expanding \(C_{ijkl}^2\), the three square terms contribute
\[
\sum_{i,j,k,l}(\delta_{ij}\delta_{kl})^2
=\sum_{i,j,k,l}\delta_{ij}\delta_{kl}
=r^2,
\]
and likewise for the other two square terms, giving \(3r^2\) in total.

For the mixed terms,
\[
\sum_{i,j,k,l}\delta_{ij}\delta_{kl}\delta_{ik}\delta_{jl}=r,
\]
since all four indices must be equal. The same holds for each of the other five mixed terms. Hence
\[
\froinner{C}{C}=3r^2+6r=3r(r+2),
\]
and therefore
\[
\froinner{\Omega}{\Omega}
=\frac{1}{r^2(r+2)^2}\,\froinner{C}{C}
=\frac{3}{r(r+2)}.
\]

Putting everything together,
\begin{align*}
\froinner{\psi(u)}{\psi(v)}
&=\froinner{u^{\otimes 4}-\Omega}{v^{\otimes 4}-\Omega} \\
&=\froinner{u^{\otimes 4}}{v^{\otimes 4}}
-\froinner{u^{\otimes 4}}{\Omega}
-\froinner{v^{\otimes 4}}{\Omega}
+\froinner{\Omega}{\Omega} \\
&=t^4-\frac{3}{r(r+2)}-\frac{3}{r(r+2)}+\frac{3}{r(r+2)} \\
&=t^4-\frac{3}{r(r+2)} \\
&=f_4^r(t).
\end{align*}

Thus, for any \(u_1,\dots,u_n\in S^{r-1}\), the matrix
\[
\bigl(f_4^r(\inner{u_i}{u_j})\bigr)_{i,j=1}^n
=
\bigl(\froinner{\psi(u_i)}{\psi(u_j)}\bigr)_{i,j=1}^n
\]
is a Gram matrix, hence is positive semidefinite.
\end{proof}

\section{Proof of the improved absolute projection constant bound}
\label{app:improved-proj-bound}

We recall the statement of \Cref{thm:improved-proj-bound}.

\medskip
\noindent\textit{\Cref{thm:improved-proj-bound}.}
Let \(r\ge4\) be even, and suppose that \(r+2\) is not a perfect square. Then,
for every rank-\(r\) orthogonal projection \(Q\in\R^{n\times n}\),
\[
    \|Q\|_1\le
    \left(
    \frac{r+\sqrt{r+2}}{1+\sqrt{r+2}}
    -
    \frac{r}{2(2r)^{4r+8}}
    \right)n.
\]
Consequently,
\[
\gamma(r)\le
\frac{r+\sqrt{r+2}}{1+\sqrt{r+2}}
-
\frac{r}{2(2r)^{4r+8}}.
\]

\medskip

The proof of \Cref{thm:projection-constant} contains more information than the
inequality itself.  In \Cref{lem:abstract-majorant}, after applying the scalar
majorant from \Cref{lem:scalar-majorant}, we discarded several nonnegative
terms.  The goal of this section is to show that, in certain dimensions, those
slack terms cannot all be small.  The scalar slack forces most pairs of row
directions to lie near the equiangular angle, while the degree-four term
prevents too much mass from concentrating on pairs with
\(|\langle u_i,u_j\rangle|\) close to \(1\).

Let
\[
    s:=\sqrt{r+2},
    \qquad
    \alpha:=\frac1s,
    \qquad
    \beta_r:=ra_r=\frac{r+s}{1+s}.
\]
For \(0\le t\le1\), define the scalar slack
\[
    h_r(t):=a_r+b_r f_2^r(t)-\rho_r f_4^r(t)-t.
\]
By \eqref{eq:direct-expansion-of-hrt}, we know that
\begin{equation}\label{eq:improved-hr-factorization}
    h_r(t)=
    \frac{(1-t)(st-1)^2(st+s+2)}{2(s+1)^2}.
\end{equation}
Thus \(h_r\ge0\) on \([0,1]\), and its only zeros are
\[
    t=1
    \qquad\text{and}\qquad
    t=\alpha=\frac1{\sqrt{r+2}}.
\]
The zero at \(\alpha\) is a double zero.  This is the common angle in the
known equiangular examples where the absolute projection constant bound is
sharp.

The proof below formalizes the following stability picture.  If a projection
nearly attains \(\|Q\|_1=\beta_r n\), then most weighted pairs of row directions
must have \(|\langle u_i,u_j\rangle|\) close to \(\alpha\), except possibly for
pairs with \(|\langle u_i,u_j\rangle|\) close to \(1\), which correspond to
repeated directions up to sign.  The scalar slack controls the first kind of bad
pair, and the degree-four term controls the second.  This produces a large
weighted graph of pairs near the angle \(\alpha\).  Weighted Tur\'an forces a
clique of size \(r+1\), but a Seidel-matrix obstruction forbids even an
approximate local clique of this type when \(r+2\) is even and not a perfect
square.

\subsection{Slack in the majorant argument}

We keep the notation from the proof of \Cref{thm:projection-constant}.  Thus
\(Q=VV^\top\), the rows of \(V\) are \(x_i=c_i u_i\), where
\[
    c_i\ge0,
    \qquad
    u_i\in S^{r-1},
    \qquad
    \sum_i c_i^2 u_i u_i^\top=I_r.
\]
Set
\[
    C:=\sum_i c_i,
    \qquad
    X:=\left\|\sum_i c_i\left(u_i u_i^\top-\frac1rI_r\right)\right\|_F,
    \qquad
    t_{ij}:=|\langle u_i,u_j\rangle|.
\]
Recall from \Cref{lem:CS-ineq} that
\begin{equation}\label{eq:improved-CS}
    C^2+rX^2\le rn.
\end{equation}
Also,
\begin{equation}\label{eq:improved-X}
    X^2=\sum_{i,j}c_i c_j\left(t_{ij}^2-\frac1r\right),
    \qquad
    \|Q\|_1=\sum_{i,j}c_i c_j t_{ij}.
\end{equation}

The following identity gives the exact slack left by the proof of
\Cref{thm:projection-constant}.

\begin{lemma}\label{lem:improved-slack-identity}
Assume \(r\ge4\).  Let
\[
    \Delta:=\beta_r n-\|Q\|_1.
\]
Then
\begin{align}\label{eq:improved-Delta-identity}
\Delta
={}&a_r\bigl(rn-C^2-rX^2\bigr)
+(ra_r-b_r)X^2  \\
&\quad
+\rho_r\sum_{i,j}c_i c_j
\left(t_{ij}^4-\frac{3}{r(r+2)}\right)
+\sum_{i,j}c_i c_j h_r(t_{ij}).\nonumber
\end{align}
Each of the four terms on the right-hand side of
\eqref{eq:improved-Delta-identity} is nonnegative.
\end{lemma}

\begin{proof}
By the definition of \(h_r\),
\[
    t=a_r+b_r f_2^r(t)-\rho_r f_4^r(t)-h_r(t)
    \qquad (0\le t\le1).
\]
Substitute \(t=t_{ij}\), multiply by \(c_i c_j\), and sum over all ordered
pairs.  Using \eqref{eq:improved-X}, we get
\[
\|Q\|_1
=a_rC^2+b_rX^2
-\rho_r\sum_{i,j}c_i c_j\left(t_{ij}^4-\frac{3}{r(r+2)}\right)
-\sum_{i,j}c_i c_jh_r(t_{ij}).
\]
Since \(\beta_r=ra_r\), rearranging gives
\eqref{eq:improved-Delta-identity}.

The first term is nonnegative by \eqref{eq:improved-CS}.  For the second, using
\(s=\sqrt{r+2}\), we have
\[
    ra_r-b_r
    =\frac{s^3+2s^2-5s-4}{2(s+1)^2}
    =\frac{(s-2)(s+1)(s+3)+2}{2(s+1)^2}>0
\]
for \(r\ge4\).  The third term is nonnegative by \Cref{lem:psd4}, and the
fourth term is nonnegative by \eqref{eq:improved-hr-factorization}.
\end{proof}

\subsection{A local Seidel obstruction}

A finite obstruction prevents the near-\(\alpha\) pairs from containing a large
clique.  The use of Seidel matrices to encode
equiangular line systems is standard. See, for example,
\cite{lemmens1973equiangular,Seidel1976SurveyTwoGraphs}.

Suppose first that there were \(r+1\) unit vectors in \(\R^r\) with
\[
    |\langle v_i,v_j\rangle|=\frac1{\sqrt{r+2}}
    \qquad (i\ne j).
\]
Choosing signs for the off-diagonal inner products, their Gram matrix has the
form
\[
    I_{r+1}+\frac1{\sqrt{r+2}}S,
\]
where \(S\) is a Seidel matrix: a symmetric matrix with zero diagonal and
\(\pm1\) off-diagonal entries.  Since \(r+1\) vectors in \(\R^r\) are linearly
dependent, this Gram matrix would be singular.  Equivalently,
\(-\sqrt{r+2}\) would be an eigenvalue of the integer matrix \(S\).

When \(r+2\) is even and not a perfect square, this is impossible.  Indeed,
\(x^2-(r+2)\) would then divide the characteristic polynomial of \(S\), but
modulo \(2\) every Seidel matrix of odd order has characteristic polynomial with
only a simple factor of \(x\).  Since \(x^2-(r+2)\equiv x^2\pmod2\), this gives
a contradiction.  The elementary root-separation argument below gives a
quantitative version of the same obstruction.

For \(m\ge2\), let \(\mathcal S_m\) denote the set of Seidel matrices of order
\(m\), let \(J_m\) denote the \(m\times m\) all-ones matrix, and write
\(\chi_S(x):=\det(xI-S)\) for the characteristic polynomial of \(S\).

\begin{lemma}\label{lem:improved-charpoly-mod2}
Let \(m\) be odd, and let \(S\in\mathcal S_m\).  Then
\[
    \chi_S(x)\equiv x(x+1)^{m-1}\pmod 2.
\]
\end{lemma}

\begin{proof}
Modulo \(2\), every off-diagonal entry of \(S\) is \(1\).  Hence
\[
    S\equiv J_m-I_m\pmod2.
\]
Over \(\mathbb F_2\), since \(m\) is odd, the matrix \(J_m-I_m\) acts as \(0\)
on \(\langle \1\rangle\) and as the identity on the hyperplane
\(\{v:\1^\top v=0\}\).  Therefore its characteristic polynomial over
\(\mathbb F_2\) is \(x(x+1)^{m-1}\).
\end{proof}

We use the following quantitative form of the Seidel obstruction.  For a
real matrix \(A\), let \(\sigma_{\min}(A)\) denote its smallest singular value.

\begin{lemma}\label{lem:improved-local-Seidel-gap}
Let \(r\ge4\), and suppose that \(d:=r+2\) is even and not a perfect square.
Then every \(S\in\mathcal S_{r+1}\) satisfies
\[
\sigma_{\min}\!\left(I_{r+1}+\frac1{\sqrt d}S\right)
\ge
\frac1{r(2r)^{2r+1}}.
\]
\end{lemma}

\begin{proof}
Since \(d=r+2\) is even, \(r+1\) is odd.  Fix
\(S\in\mathcal S_{r+1}\), and write \(p(x):=\chi_S(x)\).

We claim that \(p(-\sqrt d)\ne0\).  Indeed, suppose that
\(p(-\sqrt d)=0\).  Since \(d\) is not a perfect square, the minimal
polynomial of \(-\sqrt d\) over \(\mathbb Q\) is \(x^2-d\).  Hence
\(x^2-d\) divides \(p(x)\) in \(\mathbb Q[x]\).  Since \(x^2-d\) is monic,
and in particular primitive, Gauss's lemma implies that \(x^2-d\) divides
\(p(x)\) in \(\mathbb Z[x]\).  Reducing modulo \(2\) and using that \(d\) is
even, we get \(x^2\mid p(x)\pmod2\).  But
\Cref{lem:improved-charpoly-mod2} yields
\(p(x)\equiv x(x+1)^r\pmod2\), and the right-hand side is divisible by \(x\)
but not by \(x^2\).  This contradiction proves that
\(p(-\sqrt d)\ne0\).

Since \(p(x)\in\mathbb Z[x]\), we may write
\[
    p(\sqrt d)=u+v\sqrt d
    \qquad\text{for some }u,v\in\mathbb Z.
\]
Then \(p(-\sqrt d)=u-v\sqrt d\).  Moreover, \(p(\sqrt d)\ne0\), since
otherwise conjugation would give \(p(-\sqrt d)=0\).  Hence
\[
    p(\sqrt d)p(-\sqrt d)=u^2-v^2d\in\mathbb Z,
\]
and therefore \(|p(\sqrt d)p(-\sqrt d)|\ge1\).

Let \(\lambda_1,\ldots,\lambda_{r+1}\) be the eigenvalues of \(S\).  Since
\(S\) is a Seidel matrix, every row has absolute row sum \(r\), and so
Gershgorin's theorem~\cite[Theorem~6.1.1]{HornJohnson12} gives
\(|\lambda_i|\le r\).  Therefore
\[
    |p(\sqrt d)|
    =\prod_{i=1}^{r+1}|\sqrt d-\lambda_i|
    \le(\sqrt d+r)^{r+1}.
\]
It follows that \(|p(-\sqrt d)|\ge(\sqrt d+r)^{-(r+1)}\).

Since every factor \(|\sqrt d+\lambda_i|\) is at most \(\sqrt d+r\), we have
\[
    |p(-\sqrt d)|
    =\prod_{i=1}^{r+1}|\sqrt d+\lambda_i|
    \le(\sqrt d+r)^r\min_{1\le i\le r+1}|\sqrt d+\lambda_i|.
\]
Consequently,
\[
    \min_{1\le i\le r+1}|\sqrt d+\lambda_i|
    \ge(\sqrt d+r)^{-(2r+1)}.
\]

Finally,
\[
\sigma_{\min}\!\left(I_{r+1}+\frac1{\sqrt d}S\right)
=
\frac1{\sqrt d}\min_{1\le i\le r+1}|\sqrt d+\lambda_i|
\ge
\frac1{\sqrt d(\sqrt d+r)^{2r+1}}
\ge
\frac1{r(2r)^{2r+1}},
\]
where the last inequality uses \(r\ge4\), so that \(\sqrt d\le r\) and
\(\sqrt d+r\le2r\).
\end{proof}

For the rest of Appendix~\ref{app:improved-proj-bound}, when \(r\) is even and
\(r+2\) is not a perfect square, write
\begin{equation}\label{eq:improved-eta}
    \eta_r:=
    \frac1{r(2r)^{2r+1}}.
\end{equation}
Thus \(\eta_r\) is the explicit lower bound from
\Cref{lem:improved-local-Seidel-gap}.

The local gap rules out approximate local equiangular clusters.

\begin{corollary}
\label{cor:improved-no-local-cluster}
Suppose \(r\ge4\) is even and \(r+2\) is not a perfect square. If
$0<\varepsilon<\eta_r/r$, then there do not exist \(r+1\) unit vectors
\(v_1,\ldots,v_{r+1}\in\R^r\) such that
\[
    \left||\langle v_i,v_j\rangle|-\alpha\right|\le\varepsilon
    \qquad(i\ne j).
\]
\end{corollary}

\begin{proof}
Suppose such vectors exist, and let \(\Gamma\) be their Gram matrix.  For each
\(i\ne j\), choose the sign \(s_{ij}\in\{\pm1\}\) so that
\(s_{ij}\alpha\) has the same sign as \(\langle v_i,v_j\rangle\).  Then
\[
    \Gamma=I_{r+1}+\alpha S+E,
\]
where \(S\) is the Seidel matrix with zero diagonal and off-diagonal entries
\(s_{ij}\), while \(E\) is symmetric with zero diagonal and
\(|E_{ij}|\le\varepsilon\) for \(i\ne j\).

We claim that \(\Gamma\) is nonsingular.  Suppose instead that
\(\Gamma z=0\) for some nonzero \(z\in\R^{r+1}\).  By
\Cref{lem:improved-local-Seidel-gap},
\[
    \bigl\|(I_{r+1}+\alpha S)z\bigr\|_2\ge\eta_r\|z\|_2.
\]
On the other hand, for each \(i\), we have
\(|(Ez)_i|\le\varepsilon\sum_{j\ne i}|z_j|\).  Hence, by the
Cauchy--Schwarz inequality,
\[
\begin{aligned}
    \|Ez\|_2^2
    &\le
    \varepsilon^2\sum_i\left(\sum_{j\ne i}|z_j|\right)^2\\
    &\le
    r\varepsilon^2\sum_i\sum_{j\ne i}z_j^2
    =r^2\varepsilon^2\|z\|_2^2.
\end{aligned}
\]
Since \((I_{r+1}+\alpha S)z=-Ez\), we obtain
\[
    \eta_r\|z\|_2
    \le\|Ez\|_2
    \le r\varepsilon\|z\|_2
    <\eta_r\|z\|_2,
\]
a contradiction.  Thus \(\Gamma\) is nonsingular.  This is impossible, since
\(\Gamma\) is the Gram matrix of \(r+1\) vectors in \(\R^r\) and therefore has
rank at most \(r\).
\end{proof}

\subsection{Small defect forces many good pairs}

The next lemma converts a small normalized defect into a weighted density
estimate for the graph of pairs whose inner products are close to \(\alpha\).

\begin{lemma}\label{lem:improved-good-pairs}
Suppose that \(r\ge4\) is even and that \(r+2\) is not a perfect square.  Set
\[
    e:=\frac{\eta_r}{r}.
\]
Let \(Q\) be a rank-\(r\) projection, with \(c_i\), \(C\), and \(t_{ij}\) as
above, and set \(\Delta:=\beta_r n-\|Q\|_1\).  Let
\(I_+:=\{i:c_i>0\}\), and define \(p_i:=c_i/C\) for \(i\in I_+\).  If
\[
    \frac{\Delta}{C^2}<\frac{e^2}{128r},
\]
then the graph \(H\) with vertex set \(I_+\), in which distinct vertices
\(i,j\) are adjacent exactly when \(|t_{ij}-\alpha|\le e/2\), satisfies
\[
    2\sum_{\{i,j\}\in E(H)}p_i p_j>1-\frac1r.
\]
\end{lemma}

\begin{proof}
From \eqref{eq:improved-eta},
\[
    \frac e\alpha
    =
    \frac{\sqrt{r+2}}{r^2(2r)^{2r+1}}
    \le
    \frac1{r(2r)^{2r+1}}
    <\frac1{100}.
\]
Here we used \(r\ge4\), so that \(\sqrt{r+2}\le r\).  Since also
\(\alpha<1/2\), we have \(\alpha+e/2<1-e^2\).
Put
\[
    \theta:=1-e^2,
    \qquad
    A:=\left(\alpha-\frac e2\right)^4,
    \qquad
    D:=\theta^4-A,
    \qquad
    c:=\frac{3}{r(r+2)}.
\]
Since \(e<\alpha/100\), \(r\ge4\), and \(\alpha^2=1/(r+2)\), we have the crude
estimates
\[
    D\ge \frac45,
    \qquad
    \frac{A}{D}\le\frac1{16},
    \qquad
    \frac{c-A}{D}\le\frac2{3r}.
\]
Indeed, \(A\le\alpha^4\) and
\[
    D\ge 1-4e^2-\alpha^4
    \ge 1-\frac{1}{2500(r+2)}-\frac{1}{(r+2)^2}
    \ge \frac45.
\]
This gives \(A/D\le5/(4(r+2)^2)\le1/16\), and
\((c-A)/D\le 5c/4\le2/(3r)\).

The scalar slack has a uniform lower bound away from its two zeros.  By
\eqref{eq:improved-hr-factorization},
\[
    h_r(t)
    =
    (1-t)(t-\alpha)^2
    \frac{s^2(st+s+2)}{2(s+1)^2}
    \ge
    (1-t)(t-\alpha)^2
    \qquad(0\le t\le1),
\]
because \(s=\sqrt{r+2}\ge\sqrt6\) and
\[
    \frac{s^2(st+s+2)}{2(s+1)^2}
    \ge
    \frac{s^2(s+2)}{2(s+1)^2}
    \ge1.
\]
If \(0\le t\le\theta\) and \(|t-\alpha|\ge e/2\), then either
\(t\le(1+\alpha)/2\), in which case \(1-t\ge(1-\alpha)/2\ge1/4\), or
\(t\ge(1+\alpha)/2\), in which case
\((t-\alpha)^2\ge(1-\alpha)^2/4\ge1/16\) and \(1-t\ge e^2\).
Thus in both cases
\begin{equation}\label{eq:improved-h-lower}
    h_r(t)\ge \frac{e^2}{16}.
\end{equation}
Finally, \(\rho_r=s^3/(2(s+1)^2)\ge1/2\) for \(s=\sqrt{r+2}\ge\sqrt6\).

The weights \((p_i)_{i\in I_+}\) form a probability distribution.  Put
\[
    \mathcal E:=\frac{\Delta}{C^2}.
\]
Partition \(I_+\times I_+\) into
\[
\begin{aligned}
    \mathcal G
    &:=%
    \left\{
        (i,j):
        i\ne j,\
        |t_{ij}-\alpha|\le\frac e2
    \right\},\\
    \mathcal U
    &:=%
    \left\{
        (i,j):
        t_{ij}\ge\theta
    \right\},\\
    \mathcal B
    &:=%
    (I_+\times I_+)\setminus(\mathcal G\cup\mathcal U).
\end{aligned}
\]
The inequality \(\alpha+e/2<\theta\) ensures that
\(\mathcal G\cap\mathcal U=\varnothing\), so these three sets form a
partition.  The set \(\mathcal G\) consists of the good pairs.  Since
\(t_{ii}=1\), all diagonal pairs lie in \(\mathcal U\). Its off-diagonal pairs
correspond to nearly coincident directions up to sign.  The set
\(\mathcal B\) consists of the residual bad pairs, which are close to neither
\(\alpha\) nor \(1\).

Let
\[
    g=\sum_{(i,j)\in\mathcal G}p_i p_j,
    \qquad
    u=\sum_{(i,j)\in\mathcal U}p_i p_j,
    \qquad
    b=\sum_{(i,j)\in\mathcal B}p_i p_j.
\]
Then \(g+u+b=1\).  It therefore suffices to prove \(g>1-1/r\).

Define the normalized slack quantities:
\[
    \mathcal H=\sum_{i,j}p_i p_jh_r(t_{ij}),
    \qquad
    M=\sum_{i,j}p_i p_j
    \left(t_{ij}^4-\frac{3}{r(r+2)}\right).
\]
By \Cref{lem:improved-slack-identity} and the nonnegativity of the remaining
terms, we know that
\[
    \mathcal H+\rho_r M\le\mathcal E.
\]
In particular, \(\mathcal H\le\mathcal E\) and \(M\le2\mathcal E\).

For every \((i,j)\in\mathcal B\), we have \(t_{ij}<\theta\) and
\(|t_{ij}-\alpha|>e/2\).  Hence \eqref{eq:improved-h-lower} gives
\[
    \mathcal H\ge\frac{e^2}{16}b,
\]
and therefore
\begin{equation}\label{eq:improved-b-small}
    b\le \frac{16\mathcal E}{e^2}.
\end{equation}

The quartic term controls the mass in \(\mathcal U\).  On \(\mathcal G\), we
have \(t_{ij}^4\ge A\). On \(\mathcal U\), we have
\(t_{ij}^4\ge\theta^4\). On \(\mathcal B\), we only use \(t_{ij}^4\ge0\).
Thus
\[
    M\ge
    (A-c)g+(\theta^4-c)u-cb.
\]
Using \(g=1-u-b\), we obtain
\[
    M\ge A-c+Du-Ab.
\]
Consequently,
\[
\begin{aligned}
    1-g
    =u+b
    &\le
    \frac{c-A}{D}
    +\frac{M}{D}
    +\left(1+\frac AD\right)b\\
    &\le
    \frac2{3r}
    +\frac{5\mathcal E}{2}
    +\frac{17\mathcal E}{e^2},
\end{aligned}
\]
where we used \(D\ge4/5\), \(A/D\le1/16\), \(M\le2\mathcal E\),
\((c-A)/D\le2/(3r)\), and \eqref{eq:improved-b-small}.
Since \(\mathcal E<e^2/(128r)\) and \(e<1\),
\[
    1-g
    <
    \left(\frac23+\frac{5}{256}+\frac{17}{128}\right)\frac1r
    <\frac1r.
\]
Hence \(g>1-1/r\).  Because every unordered edge of \(H\) contributes the two
ordered pairs \((i,j)\) and \((j,i)\),
\[
    g
    =
    2\sum_{\{i,j\}\in E(H)}p_i p_j.
\]
\end{proof}

The following weighted form of Tur\'an's theorem follows from the
Motzkin--Straus theorem~\cite{MotzkinStraus1965}. See also
\cite[Theorem~1.1]{Talbot2002}.

\begin{lemma}[\cite{MotzkinStraus1965}]\label{lem:improved-weighted-turan}
Let \(H\) be a finite simple graph with vertex set \(I\), and let
\((p_i)_{i\in I}\) be nonnegative weights satisfying
\(\sum_{i\in I}p_i=1\).  If \(H\) is \(K_{r+1}\)-free, then
\[
    2\sum_{\{i,j\}\in E(H)}p_i p_j
    \le1-\frac1r.
\]
\end{lemma}

\begin{proof}[Proof of \Cref{thm:improved-proj-bound}]
Fix a rank-\(r\) orthogonal projection \(Q\in\R^{n\times n}\), and set
\(\Delta:=\beta_r n-\|Q\|_1\).  Set \(e:=\eta_r/r\), and suppose for
contradiction that \(\Delta<rn/(2(2r)^{4r+8})\).  By
\eqref{eq:improved-eta},
\[
    \frac{e^2}{128r}
    =
    \frac{1}{128r^5(2r)^{4r+2}}
    =
    \frac{r}{2(2r)^{4r+8}}.
\]
Thus \(\Delta<e^2n/(128r)\).
Hence
\[
    C^2
    \ge\sum_{i,j}c_i c_jt_{ij}
    =\|Q\|_1
    >\left(\beta_r-\frac{e^2}{128r}\right)n
    >n.
\]
Here the last inequality uses
\(\beta_r-1=(r-1)/(1+\sqrt{r+2})>1/2\) and \(e<1\).  Therefore
\[
    \frac{\Delta}{C^2}
    <
    \frac{e^2}{128r}\frac{n}{C^2}
    <
    \frac{e^2}{128r}.
\]

Let \(H\) be the weighted graph from
\Cref{lem:improved-good-pairs}, with vertex weights \(p_i:=c_i/C\) for
\(i\in I_+\).  That lemma gives
\[
    2\sum_{\{i,j\}\in E(H)}p_i p_j
    >
    1-\frac1r.
\]
Applying the contrapositive of \Cref{lem:improved-weighted-turan}, we conclude
that \(H\) contains a copy of \(K_{r+1}\).  Hence there exist distinct indices
\(i_1,\ldots,i_{r+1}\) such that
\[
    \left|
        |\langle u_{i_a},u_{i_b}\rangle|-\alpha
    \right|
    \le
    \frac e2
    =
    \frac{\eta_r}{2r}
    <
    \frac{\eta_r}{r}
    \qquad(a\ne b).
\]
Taking \(\varepsilon=e/2\) contradicts
\Cref{cor:improved-no-local-cluster}.

We therefore have
\[
    \Delta\ge\frac{r}{2(2r)^{4r+8}}n.
\]
Since \(\Delta=\beta_r n-\|Q\|_1\), it follows that
\[
    \|Q\|_1
    \le
    \left(
        \beta_r-\frac{r}{2(2r)^{4r+8}}
    \right)n
    =
    \left(
        \frac{r+\sqrt{r+2}}{1+\sqrt{r+2}}
        -
        \frac{r}{2(2r)^{4r+8}}
    \right)n.
\]
Finally, taking the maximum over all rank-\(r\) projections of each order
\(n\), and then the supremum over \(n\ge r\), gives
\[
    \gamma(r)
    \le
    \frac{r+\sqrt{r+2}}{1+\sqrt{r+2}}
    -
    \frac{r}{2(2r)^{4r+8}}.
\]
\end{proof}

\begin{remark}
The explicit improvement term in \Cref{thm:improved-proj-bound} is
intentionally crude.  The loss comes mainly from the simple estimate
\(|\lambda_i|\le r\) in the Seidel gap.  If one instead
optimizes that root-separation step using the identities
\(\sum_i\lambda_i=0\) and \(\sum_i\lambda_i^2=r(r+1)\), then one can take
\(\eta_r\) of order \((2r)^{-(r+1)}\).  Since the final gap produced by the
argument is of size \(\eta_r^2/r^3\), this would give an improvement of order
\((2r)^{-(2r+5)}\).  We do not pursue this optimization here, since the
resulting constant is still not expected to be sharp and the cruder estimate
keeps the proof short.
\end{remark}

\section{Concrete weighted cores}\label{app:weighted-examples}

We spell out the first two examples with nonuniform weights used in
\Cref{tab:computed-weighted-cores}.  In each example, the sign core \(B\) and
weights \(p\) define \(K_p=D_pBD_p\).  The projection-side value is
\(\Phi_B(p)=\sum_{j=1}^r\lambda_j(K_p)\), while the doubled graph construction
gives the graph lower bound \(\lambda_r(K_p)/2\) after taking integer blowups
with proportions tending to \(p\).  The ancillary files contain the full cores
and a verification script for the displayed values.

For \(r=4\), the finite-type extraction gives an \(11\)-vertex core with an
\(8+1+2\) weight pattern.  The numerical optimization produced this pattern.
It was not imposed.  The projection-side
value below is lower-bound evidence for \(\gamma(4)\), and we do not know
whether the construction gap between it and the graph value persists at the
extremal values.  Let
\[
v=(1,1,1,-1,-1,-1,-1,1)^\top,\qquad
\1_8=(1,\ldots,1)^\top,
\]
and
\[
R=
\begin{pmatrix}
1&1&1&1&1&-1&-1&-1\\
1&1&-1&-1&1&-1&1&1\\
1&-1&1&1&-1&1&-1&1\\
1&-1&1&1&1&1&-1&-1\\
1&1&-1&1&1&-1&1&-1\\
-1&-1&1&1&-1&1&1&1\\
-1&1&-1&-1&1&1&1&1\\
-1&1&1&-1&-1&1&1&1
\end{pmatrix}.
\]
Then
\[
B=
\begin{pmatrix}
R & -\1_8 & v & -v\\
-\1_8^\top & 1 & 1 & 1\\
v^\top & 1 & 1 & -1\\
-v^\top & 1 & -1 & 1
\end{pmatrix},
\qquad
p=(a,\ldots,a,b,c,c),
\quad 8a+b+2c=1.
\]
The projection-side weights found numerically are
\[
a_Q\approx0.0978456601,\qquad
b_Q\approx0.1072573795,\qquad
c_Q\approx0.0549886699,
\]
giving \(\Phi_B(p)\approx1.8500840104\).  On the graph side, the following exact
weights equalize the four positive eigenvalues of \(K_p\):
\[
a_G=\frac{1145-69\sqrt2}{10937},\quad
b_G=\frac{835+284\sqrt2}{10937},\quad
c_G=\frac{471+134\sqrt2}{10937},
\]
and hence
\[
c_5\ge
\frac{1007+1076\sqrt2}{10937}
\approx 0.2312054305.
\]

For \(r=5\), the construction of Der\k{e}gowska, Fickus, Foucart, and
Lewandowska \cite{DFFL22} has a \(10+6\) structure: a \(10\)-vector equiangular
tight frame in \(\R^5\) and a \(6\)-vector equiangular tight frame in \(\R^5\), arranged in a mutually unbiased
way.  The full \(16\times16\) switched sign core used for the computation
appears as \texttt{r5\_dffl\_10\_plus\_6} in the ancillary data.  The
\(6\)-vector piece receives total weight \(p^2\), and the
\(10\)-vector piece receives total weight \(q^2\), where
\[
p^2=\frac{3(8-\sqrt5)}{59},
\qquad
q^2=\frac{35+3\sqrt5}{59},
\]
so the individual class weights are \(p^2/6\) and \(q^2/10\).  The
projection-side value is
\[
\delta=\frac{5(11+6\sqrt5)}{59},
\]
and the doubled graph construction gives
\[
c_6\ge \frac{\delta}{10}
=
\frac{11+6\sqrt5}{118}
\approx0.2069187107.
\]
In this case, the numerical search produces equal lower bounds for the graph
eigenvalue constant and the projection constant after the natural scaling by
\(2r=10\).

\providecommand{\MR}[1]{}
\providecommand{\MRhref}[2]{%
  \href{http://www.ams.org/mathscinet-getitem?mr=#1}{#2}
}

\bibliography{ref}
\bibliographystyle{amsalpha}

\end{document}